\documentclass[a4paper,10 pt,reqno]{amsart}
\usepackage[latin1]{inputenc}
\usepackage[english]{babel}
\usepackage{amsmath, amsthm, amssymb, amsopn, amsfonts, amstext, stmaryrd, enumerate, color, mathtools, hyperref, mathrsfs, enumitem, url}
\usepackage[left=2.2cm,right=2.2cm,top=2.4cm,bottom=2.2cm]{geometry}
\numberwithin{equation}{section} 
\usepackage{dsfont}
\usepackage{pifont}
\usepackage{cite}

\hypersetup{colorlinks=true, 	linkcolor=blue!75!black, 	citecolor=red!50!black, 	urlcolor=green!50!black, 	linktoc=all }

\usepackage{colortbl}
\usepackage{tikz}

\newcommand{\N}{\mathbb{N}}

\newcommand{\R}{\mathbb{R}}
\newcommand{\RN}{\mathbb{R}^N}
\newcommand{\rn}{\mathbb{R}^N}

\newcommand{\dist}{{\mbox{\normalfont dist}}}

\newcommand{\vertiii}[1]{{\left\vert\kern-0.25ex\left\vert\kern-0.25ex\left\vert #1
    \right\vert\kern-0.25ex\right\vert\kern-0.25ex\right\vert}}

\newcommand{\re}{\mathbb{R}}
\newcommand{\ren}{\re^N}

\renewcommand{\a }{\alpha }

\newcommand{\s }{\sigma }

\renewcommand{\O }{\Omega }

\renewcommand{\epsilon} {\varepsilon}

\DeclareMathOperator{\diam}{diam}  

\def\XXint#1#2#3{{\setbox0=\hbox{$#1{#2#3}{\int}$ }
\vcenter{\hbox{$#2#3$ }}\kern-.6\wd0}}

\theoremstyle{plain}
\newtheorem{theorem}{Theorem}[section]
\newtheorem{proposition}[theorem]{Proposition}
\newtheorem{lemma}[theorem]{Lemma}
\newtheorem{corollary}[theorem]{Corollary}

\theoremstyle{definition}
\newtheorem{definition}[theorem]{Definition}

\theoremstyle{remark}
\newtheorem{remark}[theorem]{Remark}

\renewcommand{\le}{\leqslant}
\renewcommand{\leq}{\leqslant}

\renewcommand{\geq}{\geqslant}

\begin{document}

\title[Deterministic KPZ-type equations with nonlocal ``gradient terms'']{ Deterministic KPZ-type equations with nonlocal ``gradient terms''}

\author[B. Abdellaoui, A. J. Fern\'andez, T. Leonori and A. Younes]{Boumediene Abdellaoui, Antonio J. Fern\'andez, Tommaso Leonori and Abdelbadie Younes}

\address{
\vspace{-0.25cm}
\newline
\textbf{{\small Boumediene Abdellaoui}} \vspace{0.1cm}
\newline \indent Laboratoire d'Analyse Nonlin\'eaire et Math\'ematiques Appliqu\'ees, D\'epartement de Math\'ematiques,
\newline \indent Universit\'e Abou Bakr Belka\"id, Tlemcen, 13000, Algeria}
\email{boumediene.abdellaoui@inv.uam.es}

\address{
\vspace{-0.25cm}
\newline
\textbf{{\small Antonio J. Fern\'andez}} 
\vspace{0.1cm}
\newline \indent Instituto de Ciencias Matem\'aticas,  
\newline \indent Consejo Superior de Investigaciones Cient\'ficas, 28049 Madrid, Spain}
\email{antonio.fernandez@icmat.es}

\address{
\vspace{-0.25cm}
\newline
\textbf{{\small Tommaso Leonori}} \vspace{0.1cm}
\newline \indent Dipartimento di Scienze di Base e Applicate per l'Ingegneria,
\newline \indent Universit\'a di Roma ``Sapienza''. Via Antonio Scarpa 10, 00161 Roma, Italy}
\email{tommaso.leonori@sbai.uniroma1.it}

\address{
\vspace{-0.25cm}
\newline
\textbf{{\small Abdelbadie Younes}} \vspace{0.1cm}
\newline \indent Laboratoire d'Analyse Nonlin\'eaire et Math\'ematiques Appliqu\'ees, D\'epartement de Math\'ematiques,
\newline \indent Universit\'e Abou Bakr Belka\"id, Tlemcen, 13000, Algeria}
\email{abdelbadieyounes@gmail.com}

\maketitle
$ $

\vspace{-1cm}

\begin{abstract}
The main goal of this paper is to prove existence and non-existence results for deterministic Kardar--Parisi--Zhang type equations involving non-local ``gradient terms''. More precisely, let $\Omega \subset \RN$, $N \geq 2$, be a bounded domain with boundary $\partial \Omega$ of class $C^2$. For $s \in (0,1)$, we consider problems of the form
\[ \tag{KPZ}
\left\{
\begin{aligned}
(-\Delta)^s u & = \mu(x)\, |\mathbb{D}(u)|^q + \lambda f(x), \quad && \textup{ in } \Omega,\\
u & = 0, && \textup{ in } \RN \setminus \Omega,
\end{aligned}
\right.
\]
where $q > 1$ and  $\lambda > 0$ are real parameters, $f$ belongs to a suitable Lebesgue space, $\mu \in L^{\infty}(\Omega)$ and $\mathbb{D}$ represents a nonlocal ``gradient term''. Depending on the size of $\lambda > 0$, we derive existence and non-existence results. In particular, we solve several open problems posed in 
\cite[Section 6]{AP-KPZ-2018} and \cite[Section 7]{AF-2020}.

\bigbreak

\smallbreak
\medbreak
\end{abstract}

\section{Introduction}
\noindent In this paper we analyse the existence and non--existence of solutions for deterministic Kardar--Parisi--Zhang type equations involving non-local ``gradient terms''. For $s, t \in (0,1)$, we consider problems of the form
\[ \tag{KPZ} \label{KPZ}
\left\{
\begin{aligned}
(-\Delta)^s u & = \mu(x)\, |\mathbb{D}_t(u)|^q + \lambda f(x), \quad && \textup{ in } \Omega,\\
u & = 0, && \textup{ in } \RN \setminus \Omega,
\end{aligned}
\right.
\]
depending on a real parameter $\lambda > 0$. Here, $\Omega \subset \RN$, $N \geq 2$, is a bounded domain with boundary $\partial \Omega$ of class $C^2$, $f$ belongs to a suitable Lebesgue space, $\mu \in L^{\infty}(\Omega)$, $q \in (1,+\infty)$ and $\mathbb{D}_t$ represents one of the following nonlocal ``gradient terms'' :
\begin{align} 
& \circ\quad  (-\Delta)^{\frac{t}{2}}u(x) :=  a_{N,\frac{t}{2}} \, {\rm{P.V.}} \int_{\RN} \frac{u(x)-u(y)}{|x-y|^{N+t}} dy  && (\textup{Half $t$--Laplacian}), \tag{KPZ$_1$} \label{KPZ1-intro} \\
& \circ\quad  \nabla^{t} u (x) := \mu_{N,t} \int_{\RN} \frac{(x-y)(u(x)-u(y))}{|x-y|^{N+t+1}} dy   && ( \textup{Riesz $t$--Gradient}), \tag{KPZ$_2$} \label{KPZ2-intro} \\
& \circ\quad \mathcal{D}_t u (x) := \left( \frac{a_{N,t}}{2} \int_{\RN} \frac{(u(x)-u(y))^2}{|x-y|^{N+2t}}dy. \right)^{\frac12} && (\textup{Stein $t$--Functional}). \tag{KPZ$_3$} \label{KPZ3-intro}
\end{align}
Note that the previous definitions make sense for any function $u \in C_c^{\infty}(\RN)$. Also, let us point out that
\[ a_{N,\sigma} := -\, \frac{ 2^{2\sigma} \Gamma \left( \frac{N}{2}+\sigma \right)}{ \pi^{\frac{N}{2}} \Gamma(-\sigma) } \qquad  \textup{ and }  \qquad \mu_{N,\sigma} := \frac{2^{\sigma} \Gamma( \frac{N+\sigma+1}{2} )}{\pi^{\frac{N}{2}} \Gamma( \frac{1-\sigma}{2})}, 
\]
are normalization constants and ``P.V.'' stands for ``in the principal value sense''. Since both these constants and the ``principal value sense'' will not play an important role in our work, we will omit them from now on. 

Before going further, we would like to emphasize that the three different nonlocal ``gradient terms'' that we consider can be traced back many years ago. Since nowadays the \textit{fractional Laplacian} does not need any further presentation, let us focus on the other two terms. As very well explained in \cite[page 3]{SS-2015}, the origin of the \textit{Riesz $t$--Gradient} seems to be \cite{R-49}. Note also that this operator has been rediscovered several times since \cite{R-49} and has received considerable attention in the last few years. See for instance \cite{S-20, SSV-S-2017, MS-2015, CS-2019}. On the other hand, the \textit{Stein $t$--Functional} can be at least traced back to \cite{S-61}. Moreover, this operator naturally appears as the nonlocal equivalent to the gradient when considering the minimization of fractional Harmonic maps into the sphere. See for instance the recent papers \cite{CD-2018, S-2015, MillS-2015}. 

In contrast with the local case
\begin{equation} \label{local}
\left\{
\begin{aligned}
-\Delta u & = \mu(x) |\nabla u|^q + \lambda f(x), && \textup{ in } \Omega,\\
u & = 0, && \textup{ on } \partial \Omega,
\end{aligned}
\right.
\end{equation}
where the literature is very extensive, there exist very few results dealing with equations of the form \eqref{KPZ}. We refer to \cite{AF-2020}, by the first two authors, for a detailed introduction to the subject. However, since the publication of \cite{AF-2020}, some results have been published. For instance, we would like to mention the recent paper \cite{AC2020}, where the authors establish existence and non--existence results for problems of the form \eqref{KPZ} with a local operator (the Laplacian instead of the fractional Laplacian) and a nonlocal nonlinearity. Also, let us mention \cite{BM-2021}, where the authors establish the equivalence between different notions of solution to problems of the form \eqref{KPZ}.
Finally, let us stress that solutions to equations with nonlocal diffusion and a nonlocal \lq\lq gradient term\rq\rq{ } with nonsingular kernels have been studied in \cite{LMS}.

In \cite{AF-2020}, the first two authors analyse the existence and non--existence of solutions to \eqref{KPZ}, under the additional assumption $t = s \in (1/2,1)$. The main goal of this paper is to refine the approach of \cite{AF-2020} in order to deal also with the cases where $s \in (0,1/2]$ and/or $t \neq s$. Depending on the real parameter $\lambda > 0$, we analyse the existence and non--existence of weak solutions to \eqref{KPZ} under the assumptions
\begin{equation} \tag{A1} \label{A1}
\left\{
\begin{aligned}
\ & q \in (1,+\infty),  \\
&  0< t < \min\{1,s(1+(qN)^{-1})\},\\
& f \in L^m (\Omega)\ \textup{ for some } m > N/s\ \textup{ and }\  \mu \in L^{\infty} (\Omega).
\end{aligned}
\right.
\end{equation}
Following \cite{C_V_2014-JFA,C_V_2014-JDE}, we introduce the subsequent notion of weak solution to \eqref{KPZ} :

\begin{definition} \label{weak Sol KPZ1} We say that $u$ is a\textit{ weak solution} to \eqref{KPZ} if $u$ and $|\mathbb{D}_t(u)|^q$ belong to $L^1(\Omega)$, $u \equiv 0$ (a.e.) in $\RN \setminus \Omega$ and
\begin{equation}
\int_{\Omega} u (-\Delta)^s \phi\, dx = \int_{\Omega} \big( \mu(x)|\mathbb{D}_t(u)|^q + \lambda f(x) \big) \phi\, dx,
\end{equation}
for all $\phi$ belonging to
\begin{equation*} 
\mathbb{X}^s(\Omega) := \Big\{\phi\in C^s(\ren)\,: \, \phi (x) = 0 \textup{ for all }x \in \RN \setminus \Omega \textup{ and } (-\Delta)^s \phi \in L^{\infty}(\Omega)\Big\}.
\end{equation*} 
\end{definition}

\begin{remark} With a slight abuse of notation we use {\rm(KPZ$_{\,i}$)}, $i = 1,2,3$, to refer to \eqref{KPZ} with $\mathbb{D}_t = (-\Delta)^{\frac{t}{2}}$, $\mathbb{D}_t = \nabla^t$ and $\mathbb{D}_t = \mathcal{D}_t$ respectively. See \eqref{KPZ1-intro}, \eqref{KPZ2-intro} and \eqref{KPZ3-intro} for the corresponding definitions.
\end{remark}

\medbreak
Our main existence result can be informally stated as follows:
 
\begin{theorem} \label{exist-intro}
Assume that \eqref{A1} holds true.
Then, there exists $\lambda^{\star} > 0$ such that, for all
$0 < \lambda \leq \lambda^{\star}$, {\rm(KPZ$_{\,i}$)}, $i = 1,2,3$, has a weak solution $u$. Moreover, $u \in W^{s,p}(\RN) \cap C^{0,s}(\RN)$ for all $1 < p < +\infty$. 
\end{theorem}


Theorem \ref{exist-intro} is a particular case of the more general existence results proved in Section \ref{existenceResults}. We refer directly to Section \ref{existenceResults} for more general statements. In particular, let us emphasize that, in Section \ref{existenceResults}, we substantially lower the regularity imposed on the data $f$. Furthermore, arguing as in the proof of \cite[Theorem 1.3]{AF-2020}, it is possible to show that the regularity considered in Section \ref{existenceResults} is almost optimal. Note also that our existence results solve several open problems posed in \cite{AF-2020, AP-KPZ-2018}.

The proofs of our existence results rely on the combination of fixed point arguments in the spirit of \cite{P-2014, MP-2016} with \textit{global} fractional Calder\'on--Zygmund regularity results for the fractional Poisson equation 
\begin{equation} \label{Poisson}
\left\{
\begin{aligned}
(-\Delta)^s u & = h, \quad && \textup{ in } \Omega,\\
u & = 0, && \textup{ in } \RN \setminus \Omega.
\end{aligned}
\right.
\end{equation}
This approach was already implemented in \cite{AF-2020} by the first two authors. However, the required \textit{global} fractional Calder\'on--Zygmund regularity theory was not available for $s \in (0,1/2]$. Furthermore, the Calder\'on--Zygmund regularity used in \cite{AF-2020} (cf. \cite[Section 3]{AF-2020}) contains several imprecisions. In the recent paper \cite{AFLY-2021}, we establish sharp \textit{global} fractional Calder\'on--Zygmund regularity results for \eqref{Poisson} in the full range $s \in (0,1)$ (fixing in particular the issues of \cite[Section 3]{AF-2020}). Having at hand these regularity results, the proofs of our existence results follow from a refinement of the fixed point approach implemented in \cite{AF-2020}. We refer to Section \ref{existenceResults} for more details. 

Taking into account Theorem \ref{exist-intro}, it is very natural to ask whether the smallness assumption in $\lambda$ is necessary or not. Note that, in \cite[Theorem 1.2]{AF-2020}, the first two authors established a non--existence result for (KPZ$_{3}$) with $t = s \in (0,1)$. We focus here in proving a non-existence result to (KPZ$_1$). This was left as an open problem in \cite[Section 7]{AF-2020}. 

\begin{theorem} \label{nonexist-intro}
Assume that \eqref{A1} holds true with $t = s$ and $q = 2$ and suppose that $\mu(x) \geq \mu_1 > 0$ and $f \gneqq 0$. Then, there exists $\lambda^{\star \star} > 0$ such that, for all
$ \lambda > \lambda^{\star \star}$, {\rm(KPZ$_{1}$)} has no weak solution in $W^{s,2}(\rn)$.
\end{theorem}

\begin{remark} 
The non--existence for $\lambda$ large in the case where $\mathbb{D}_t = \nabla^t$ remains completely open and a different approach is missing. 
\end{remark}


\subsection*{Organization of the paper} In the next section we introduce the main function spaces involved in our results and prove the continuity and compactness of the solution map for the fractional Poisson equation with $L^1$--data. In Section \ref{existenceResults}, we prove our main existence results to \eqref{KPZ}, from which Theorem \ref{exist-intro} immediately follows. Finally, in Section \ref{nonExistenceResults}, we prove non-existence results for \eqref{KPZ1} and \eqref{KPZ3} when $\lambda > 0$ is large. 

\subsection*{Acknowledgements.} 
This work has received funding from the European Research Council (ERC)
under the European Union's Horizon 2020 research and innovation programme
through the Consolidator Grant agreement 862342 (A. J. F.). B. A. is partially supported by projects MTM2016-80474-P and  PID2019-110712GB-I00, MINECO, Spain. B. A. and A. Y. are partially supported by the DGRSDT, Algeria. Part of this work was done while A. J. F. was visiting the Universit\`a di Roma ``La Sapienza'' and the Universit\`a di Bologna. He thanks his hosts for the kind hospitality and the financial support. The authors wish to thank the anonymous referee for the very useful remarks, which helped to improve the presentation of the paper.

\section{Function spaces and tools}
\noindent We collect here the definition of the main function spaces involved in our results and some other tools. First of all, recall that, for all $s \in (0,1)$ and $ 1 \leq p < +\infty$, the fractional Sobolev space $W^{s,p}(\RN)$ is defined as
$$
W^{s,p}(\RN):= \left\{ u \in L^p(\RN) : \iint_{\R^{2N}} \frac{|u(x)-u(y)|^p}{|x-y|^{N+sp}} dx dy < +\infty \right\}.
$$
It is a Banach space endowed with the usual norm
$$
\|u\|_{W^{s,p}(\RN)} = \left( \|u\|_{L^p(\RN)}^p + \iint_{\R^{2N}} \frac{|u(x)-u(y)|^p}{|x-y|^{N+sp}} dx dy \right)^{\frac{1}{p}}.
$$
Also, having at hand $W^{s,p}(\RN)$, we define the space $W_0^{s,p}(\Omega)$ as
\[ W_0^{s,p}(\Omega) := \left\{ u \in W^{s,p}(\ren): u \equiv 0 \textup{ in } \ren \setminus \Omega \right\},\]
and recall that, thanks to the Sobolev inequality, it is a Banach space endowed with the norm
\[ \|u\|_{W^{s,p}_0(\Omega)} := \left( \iint_{D_{\Omega}} \frac{|u(x)-u(y)|^p}{|x-y|^{N+sp}} dx dy \right)^{1/p}, \quad \textup{ where } \quad  D_{\Omega} := (\Omega \times \ren) \cup ((\RN \setminus \Omega) \times \Omega).\]

\noindent Next, we remind that, for any $s \in (0,1)$ and $1 \leq p < +\infty$, the Bessel potential space  is defined as 
$$
L^{s,p}(\RN) \ :=   \ \overline{ \big\{u \in C_c^{\infty}(\RN)\big\} }^{\ \pmb{|}\cdot\pmb{|}_{L^{s,p}(\RN)}},
$$
where 
$$
\pmb{|} u \pmb{|}_{L^{s,p}(\RN)} = \|(1-\Delta)^{\frac{s}{2}}u\|_{L^p(\RN)} \quad \textup{ and } \quad (1-\Delta)^{\frac{s}{2}}u = \mathcal{F}^{-1} ( (1+|\cdot|^2)^{\frac{s}{2}} \mathcal{F} u), \quad \forall\ u \in C_c^{\infty}(\RN).
$$
Let us stress that, in the case where $s \in (0,1)$ and $1 < p < +\infty$,
$$
\|u\|_{L^{s,p}(\RN)} := \|u\|_{L^p(\RN)} + \|(-\Delta)^{\frac{s}{2}}u\|_{L^p(\RN)}
$$
is an equivalent norm for $L^{s,p}(\RN)$ (see e.g. \cite[Theorem 2]{S-61}). By \cite[Theorem 1.1]{S-61}, we also know that, if in addition $2N/(N+2s) < p < + \infty$, then $L^{s,p}(\RN)$ can be equipped with the equivalent norm
$$
\vertiii{u}_{L^{s,p}(\RN)} :=  \|u\|_{L^p(\RN)} + \|\mathcal{D}_s(u)\|_{L^p(\RN)}.
$$
In analogy with $W_0^{s,p}(\Omega)$, let us define
$$
L_0^{s,p}(\Omega):= \big\{ u \in L^{s,p}(\RN) : u \equiv 0 \textup{ in } \RN \setminus \Omega \big\},
$$
and stress that, if $0 < s < 1$ and $1 < p < +\infty$, it is a Banach space endowed with the norm
$$
\|u\|_{L^{s,p}_0(\Omega)} := \|(-\Delta)^{\frac{s}{2}} u \|_{L^p(\RN)}.
$$
If in addition $2N/(N+2s) < p < + \infty$, then $L_{0}^{s,p}(\Omega)$ can also be equipped with the equivalent norm
$$
\vertiii{u}_{L^{s,p}_0(\Omega)} := \|\mathcal{D}_s(u)\|_{L^p(\RN)}.
$$

\noindent Let us as well recall that, for all $0 < \epsilon < \sigma < 1$ and all $1 < p < +\infty$, by \cite[Theorem 7.63, (g)]{adams}, we have
$$
L^{\sigma+\epsilon,p}(\RN) \subset W^{\sigma,p}(\RN) \subset L^{\sigma-\epsilon,p}(\RN).
$$
It is also well known that, for all $1 \leq p < +\infty$ and all $0 < \sigma \leq \sigma' < 1$,
$$
W^{\sigma',p}(\RN) \subset W^{\sigma,p}(\RN),
$$
and that (cf. \cite[Theorem 7.63 (c)]{adams}), if in addition $1 < p < + \infty$, 
$$
L^{\sigma',p}(\RN) \subset L^{\sigma,p}(\RN).
$$
Since the constants will be useful later on, let us emphasize there exists $\widetilde{k} := \widetilde{k} (\sigma,\sigma',p) \geq 1$ such that
\begin{equation}\label{emb}
\|u\|_{L^{\sigma,p}(\RN)} \leq \widetilde{k} \|u\|_{L^{\sigma',p}(\RN)}, \quad \textup{ for all } u \in L^{\sigma',p}(\RN),
\end{equation}
and
\begin{equation} \label{emb2}
\|u\|_{W^{\sigma,p}(\RN)} \leq \widetilde{k} \|u\|_{W^{\sigma',p}(\RN)}, \quad \textup{ for all } u \in W^{\sigma',p}(\RN),
\end{equation}

\medbreak
Now, for $0 < s \leq t < \min\{1,s(1+N^{-1})\}$, let us set\footnote{We use,  here and in the sequel,  the convention $1/a^{+} = +\infty$ if $a \leq 0$.}
\begin{equation*}
\widetilde{p}\,(m,s,t):= 
\left\{
\begin{aligned}
& \frac{1}{(t-s)^{+}}, \quad && \textup{ if } m > \frac{N}{2s-t},\\
& \min\Big\{\frac{mN}{N-ms+mN(t-s)}, \frac{1}{(t-s)^{+}} \Big\}, \quad && \textup{ if } 1 \leq m < \frac{N}{2s-t}.
\end{aligned}
\right.
\end{equation*} 
In our next result, which will be very useful in the sequel, we describe the regularity of the (unique) solution to \eqref{Poisson}. Note that such a  result is contained in  \cite[Theorems 1.3 and 5.2,  Corollaries 5.3 -- 5.6]{AFLY-2021}. 

\begin{proposition} \label{CZ}
Let $0<s\leq t < \min\{1,s(1+N^{-1})\}$ and let $u$ be the (unique) weak solution to \eqref{Poisson} with $h \in L^{m}(\Omega)$ for some $m\geq 1$.
Then, for all $1 < p < \widetilde{p}$, there exists $\widetilde{C}(N,s,p,m,\Omega) > 0$ such that
\begin{equation} \label{CZeq}
 \|u\|_{L^{t,p}(\RN)} \leq  \widetilde{C} \, \|h\|_{L^m(\Omega)} \quad \textup{ and } \quad \|u\|_{W^{t,p}(\RN)} \leq \widetilde{C} \, \|h\|_{L^m(\Omega)}.
\end{equation}
\end{proposition}

We also present here a technical but useful lemma proved in \cite[Lemma 5.1]{AFLY-2021} and a classical result from harmonic analysis.

\begin{lemma} \label{declem}
Let $s \in (0,1)$, $s \leq t < \min\{1,2s\}$ and and let $u$ be the (unique) weak solution to \eqref{Poisson}  with $h \in L^1(\Omega)$. Then, there exists $C := C(N, s, t,  \Omega) > 0$  such that
\begin{equation} \label{deco}
\big| (-\Delta)^{\frac{t}{2}} u (x)| \leq C \left[ g_1(x) + |\log \delta(x)| g_2(x) + \frac{1}{\delta^{t-s}(x)} g_3(x) \right], \quad \mbox{ for a.e. } x \in \Omega.
\end{equation}
Here, $\delta(x) := \dist(x,\partial \Omega)$ and the functions $g_i$, $i = 1,2,3$, satisfy: 
\begin{align} \label{g1}
& \bullet\ \mbox{For all $0 < \lambda < 2s-t$, there exists $C:= C(\lambda) > 0$ s.t.
$ \displaystyle
g_1(x) \leq C \int_{\Omega} \frac{|f(y)|}{|x-y|^{N-(2s-t-\lambda)}} dy, 
$} \\
& \bullet\ g_2(x) := \int_{\Omega} \frac{|f(y)|}{|x-y|^{N-(2s-t)}} dy \label{g2} , \\
& \bullet\  g_3  (x) := (t-s) \int_{\Omega} \frac{|f(y)|}{|x-y|^{N-s}} dy.  \label{g3}
\end{align}
Moreover, for any $R \geq \frac{1}{3} + \frac{4}{3} (\diam(\Omega) + \dist(0,\Omega))$, it follows that
\begin{equation} \label{med}
\begin{aligned}
\big| (-\Delta)^{\frac{t}{2}} u(x) \big| 
\leq \int_{\Omega} \frac{|u(y)|}{\delta^s(y)} \frac{dy}{|x-y|^{N+(t-s)}}  , \quad \textup{ for a.e. } x \in B_R(0) \setminus \Omega\,,
\end{aligned}
\end{equation} 
and
\begin{equation} \label{far}
\big| (-\Delta)^{\frac{t}{2}} u(x) \big| \leq \frac{4^{N+t}}{(1+|x|)^{N+t}} \int_{\Omega} |u(y)| dy,
 \quad \textup{ for a.e. } x \in \RN \setminus B_R(0).
\end{equation}
\end{lemma}
 
\begin{lemma}\label{Stein}  
Let $\omega \subset \RN$, $N \geq 2$, be an open bounded domain and let $0 < \alpha < N$ and  $1\le p<\ell<\infty$ be such that  $\dfrac{1}{\ell} =\dfrac{1}{p}-\dfrac{\a}{N}$. Moreover, for $g\in L^p(\omega)$, let 
$$
J_\lambda(g)(x): =\int_{\omega} \dfrac{g(y)}{|x-y|^{N-\a}}dy\,.
$$
It follows that there exists  $C= C(N,\a, p , \s > 0, \ell, \O)>0 $  such that 
\begin{itemize}
\item[$a)$] $J_{\alpha}$ is well defined (in the sense that the integral converges absolutely for a.e. $x \in \omega$).
\item[$b)$] $[J_{\alpha}(g)]_{M^{\ell}(\omega)} \leq C\, \|g\|_{L^1(\omega)}.$  In particular, $\|J_{\alpha}(g)\|_{L^{\sigma}(\omega)} \leq C \, \|g\|_{L^1(\omega)}$ for all $1 \leq \sigma < \ell$. 
\item[$c)$] If $1<p<\frac{N}\a$, then $\|J_{\alpha}(g)\|_{L^\ell (\omega)} \leq C\, \|g\|_{L^p(\omega)}$. 
\item[$d)$] If $p = \frac{N}{\alpha}$, then $\|J_{\alpha}(g)\|_{L^{\sigma}(\omega)} \leq C \, \|g\|_{L^p(\omega)}$ for all $1 \leq \sigma < +\infty$.
    \item[$e)$]  If $p>\frac{N}\a$, then $
\|J_{\alpha}(g)\|_{L^{\infty}(\omega)} \leq C\, \|g\|_{L^p(\omega)}.
$  
\end{itemize}
\end{lemma}

\begin{proof} 
Parts a), b), c) and d) follow from \cite[Theorem I, Section 1.2, Chapter V, page 119]{Stein}. Part e) is contained in \cite[Lemma 7.12]{G_T_2001_S_Ed} (see also \cite[Theorem 2.2]{mizuta}).
\end{proof}

\bigbreak

We conclude this section proving a compactness result for the fractional Poisson equation \eqref{Poisson} that will be key in the proof of our main existence results. Here, we denote by   $G_s : \R_{\ast}^{2N} \to \R$ the Green function associated to $(-\Delta)^{s}$ in $\Omega$ with homogeneous Dirichlet boundary conditions. Note that $\R_{\ast}^{2N}:= \{(x,y) \in \R^{2N}: x \neq y \}$.

\begin{proposition} \label{solMap}
Let $0 < t < \min\{1,s(1+N^{-1})\}$ and $1 < p < N/(N(1+t-s)-s)$. The solution map 
$$
\mathbb{G}_s : L^1(\Omega) \to L_0^{t,p}(\Omega), \qquad h \mapsto \mathbb{G}_s[h] := \int_{\Omega} G_s(x,y) h(y)\, dy, 
$$ 
is well-defined, continuous and compact. 
\end{proposition}

\begin{proof}
First of all, note that without loss of generality we can assume that $s \leq t < \min\{1,s(1+N^{-1})\}$. Indeed, once we have the result in this case, we can infer the result for $t \in (0,s)$ interpolating as in the proof of \cite[Proposition 3.10]{AF-2020}. Let us also stress that, by Proposition \ref{CZ}, the solution map $\mathbb{G}_s$ is well-defined and continuous for all $1 < p < N/(N(1+t-s)-s)$. Hence, we just have to show that $\mathbb{G}_s$ is compact for the same range of $p$. 

Let $(h_n)_n \subset L^1(\Omega)$ be a sequence such that $\|h_n\|_{L^1(\Omega)} \leq 1$ for all $n \in \N$ and let $u_n = \mathbb{G}_s(h_n)$ for all $n \in \N$. 
By \cite[Proposition 2.6]{C_V_2014-JDE} we know that, up to a subsequence, $(u_n)_n$ is strongly convergent in $L^q(\Omega)$ for all $1 \leq q < N/(N-2s)$. Moreover, combining this strong convergence with Vitali's convergence theorem, we deduce that, up to a subsequence, $(u_n/\delta^{s})_n$ is strongly convergent in $L^r(\Omega)$ for all $1 \leq r < N/(N-s)$. Having at hand the (up to a subsequence) strong convergence of $(u_n)_n$ in $L^q(\Omega)$ for all $1 \leq q < N/(N-2s)$, to end the proof, we just have to show that, up to a subsequence, $((-\Delta)^{\frac{t}{2}}u_n)_n$ is strongly convergent in $L^p(\RN)$ for all $1 < p < N/(N(1+t-s)-s)$.

First,  for $0 < \alpha < 2$, let us  consider the integral operator
$$
T_{\alpha}: L^1(\Omega) \to L^{\gamma}(\Omega), \qquad g \mapsto \int_{\RN} \frac{g(y) \mathds{1}_{\Omega}(y)}{|x-y|^{N-\alpha}} dy.
$$
Note that, by \cite[Lemma 7.12]{G_T_2001_S_Ed}, $T_{\alpha}$ is well-defined and continuous for all $1 \leq \gamma < N/(N-\alpha)$. Moreover, following step by step the proof of \cite[Theorem 2.2]{SSII2018}, one can prove that $T_{\alpha}$ is compact for all $1 \leq \gamma < N/(N-\alpha)$. We only sketch the proof. Let $(g_n)_n \subset L^1(\Omega)$ be a sequence such that $\|g_n\|_{L^1(\Omega)} \leq 1$ for all $n \in \N$, $v_n := T_{\alpha}(g_n)$ for all $n \in \N$ and $1 \leq \gamma < N/(N-\alpha)$. If we prove the existence of a  (not relabelled) subsequence $(v_n)_n$ that is Cauchy in $L^{\gamma}(\Omega)$, the compactness immediately follows. Let $\eta_{\epsilon}$ be a standard mollifier and let $v_n^{\epsilon}:= v_n \star \eta_{\epsilon}$ for all $\epsilon > 0$ and all $n \in \N$. Following \cite[Theorem 2.2]{SSII2018}, we obtain that, for all $1 \leq \gamma < N/(N-\alpha)$, there exist constants $C > 0$ and $\sigma > 0$ (independent of $\epsilon$ and $n$) such that
\begin{equation} \label{aa}
\|v_n^{\epsilon}- v_n\|_{L^{\gamma}(\Omega)} \leq C \epsilon^{\sigma}.
\end{equation}
On the other hand, using the standard properties of the mollifiers $\eta_{\epsilon}$, we get that, for any $\epsilon > 0$, the sequence $(v_n^{\epsilon})_n$ is bounded and equicontinuous in $C(\RN)$. Thus, Arzel\`a-Ascoli Theorem implies the existence of a subsequence uniformly convergent in $\overline{\Omega}$. Combining \eqref{aa} with the uniform convergence in $\overline{\Omega}$, a standard diagonal argument shows the existence of a (not relabelled) Cauchy subsequence $(v_n)_n$ in $L^{\gamma}(\Omega)$, as desired.

Having at hand the compactness of $T_{\alpha}$ and Lemma \ref{declem} we now prove that, up to a subsequence, $((-\Delta)^{\frac{t}{2}}u_n)_n$ is strongly convergent in $L^{p}(\RN)$ for all $1 \leq p < N/\big(N(1+t-s)-s\big)$. To that end, we fix  $R \geq  \frac13 + \frac43 \big(\diam(\Omega) + \dist(0,\Omega)\big)$ as in Lemma \ref{declem} and
split $\RN$ in three regions: $\Omega$, $B_R(0) \setminus \Omega$ and $\RN \setminus B_R(0)$. 

Combining \eqref{deco} with the linearity of the problem \eqref{Poisson}, the compactness of $T_{\alpha}$ and H\"older inequality, we get that, up to a subsequence, $((-\Delta)^{\frac{t}{2}} u_n)_n$ is strongly convergent in $L^{p}(\Omega)$ for all $1 \leq p < N/\big(N(1+t-s)-s\big)$. 

Next, we deal with the strong convergence of 
$((-\Delta)^{\frac{t}{2}}u_n)_n$ in $L^{p}(B_R(0) \setminus \Omega)$ for all $1 \leq p < N/(N(1+t-s)-s)$. Since $|x-y| \geq \max\{\delta(y),\delta(x)\}$ for all $y \in \Omega$ and $x \in B_R(0) \setminus \Omega$, by \eqref{med} we have  that, for all $\epsilon > 0$,  
$$
\big| (-\Delta)^{\frac{t}{2}} u_n (x) \big| \leq \int_{\Omega} \frac{|u_n (y)|}{|x-y|^{N+t}} dy \leq \frac{1}{\delta^{t-s+\epsilon}(x)} \int_{\Omega} \frac{|u_n (y)|}{\delta^s(y)} \frac{dy}{|x-y|^{N-\epsilon}} \quad \textup{ for a.e. } x \in \RN \setminus B_R(0)\,.
$$
Combining this inequality with Lemma \ref{Stein}, H\"older inequality and the (up to a subsequence) strong convergence of $(u_n/\delta^{s})_n$ in $L^r(\Omega)$ for all $1 \leq r < N/(N-s)$, we get the desired convergence in $L^p(B_R(0) \setminus \Omega)$. 

Finally, the (up to a subsequence) strong convergence of 
$((-\Delta)^{\frac{t}{2}}u_n)_n$ in $L^{p}(\RN \setminus B_R (0))$ for all $1 \leq p < N/(N(1+t-s)-s)$ follows from \eqref{far}. Indeed, combining \eqref{far} with the (up to a subsequence) strong convergence of $(u_n)_n$ in $L^q(\Omega)$ for all $1 \leq q < N/(N-2s)$, we get the desired convergence in $L^p(\RN \setminus B_R(0))$.
\end{proof}

\begin{corollary} \label{solMapSob}
Let $0 < t < \min\{1,s(1+N^{-1})\}$ and $1 < p < N/(N(1+t-s)-s)$. The solution map 
$$
\mathbb{G}_s : L^1(\Omega) \to W_0^{t,p}(\Omega), \qquad h \mapsto \mathbb{G}_s[h] := \int_{\Omega} G_s(x,y) h(y)\, dy, 
$$ 
is well-defined, continuous and compact. 
\end{corollary}

\begin{proof}
Having at hand Proposition \ref{solMap} the result follows arguing as in the proof of \cite[Corollary 5.6]{AFLY-2021}
\end{proof}

 \begin{remark}  $ $
\begin{itemize}
\item[$\circ$]
In the case where $1/2 < t < \min\{1,s(1+N^{-1})\}$, Proposition \ref{solMap} and Corollary \ref{solMapSob} can be proved arguing as in the proof of \cite[Proposition 3.10]{AF-2020}.
\item[$\circ$] We believe Proposition \ref{solMap} and Corollary \ref{solMapSob} are of independent interest and will be useful elsewhere. 
\end{itemize}
\end{remark}

\section{Existence results} \label{existenceResults}

\noindent This section is devoted to prove existence results for \eqref{KPZ} with the different choices of $\mathbb{D}_t$ present in the introduction. We will analyse in parallel $\mathbb{D}_t = (-\Delta)^{\frac{t}{2}}$ and $\mathbb{D}_t = \nabla^t$ and separately $\mathbb{D}_t = \mathcal{D}_t$. Let us emphasize that the main existence result stated in the introduction, namely Theorem \ref{exist-intro}, immediately follows from the results of this section. 

We first analyse the case $\mathbb{D}_t = (-\Delta)^{\frac{t}{2}}$ (the case $\mathbb{D}_t = \nabla^t$ follows arguing on the exact same way, as we will detail later on). More precisely, for $q \in (1,+\infty)$, we analyse the existence of weak solution to
\begin{equation}  \tag{KPZ$_1$}\label{KPZ1} 		
\left\{ 		
\begin{aligned} 			
(-\Delta)^s u & = \mu(x) |(-\Delta)^{\frac{t}{2}}u|^q + \lambda f(x), && \quad \textup{ in } \Omega,\\ 
u & = 0, && \quad \textup{ in }   \RN \setminus \Omega.\\
\end{aligned} 		
\right. 	
\end{equation}




\noindent Let us impose $0 < t < \min\{1,s(1+N^{-1})\}$ and set
\begin{equation}
\overline{q}(m,s,t):= \left\{ \begin{aligned}
& +\infty, && \mbox{ if } t\leq s\  \mbox{ and }\ m \geq  N/s,\\
&  s/(N(t-s)), \qquad &&\mbox{ if } t>s\ \mbox{ and }\ m > N/s,\\
& N/(N-ms), && \mbox{ if } t \leq s \ \mbox{ and } 1 \leq m < N/s,\\
& N/(N-sm + mN(t-s)),\quad &&\mbox{ if } t > s \ \mbox{ and }1 \leq m \leq N/s. 
\end{aligned}
\right.
\end{equation}
Having at hand $\overline{q} \in (1,+\infty]$, our main result concerning \eqref{KPZ1} reads as follows:

\begin{theorem} \label{KPZ1existence}
Assume that $0 < t < \min\{1,s(1+N^{-1})\}$, $f \in L^m(\Omega)$ for some $m \geq 1$ and $\mu \in L^{\infty}(\Omega)$. Then, for all $1 < q < \overline{q}$, there exists $\lambda_{\star} > 0$ such that, for all $0 < \lambda \leq \lambda_{\star}$, \eqref{KPZ1} has a weak solution $u$. Moreover:
\begin{itemize}
\item[$\circ$] If $m \geq N/s$, then $u \in W^{s,p}(\RN) \cap C^{0,s}(\RN)$ for all $1 < p < + \infty$. 
\item[$\circ$] If $1 \leq m < N/s$, then $u\in W^{s,p}(\RN)$ for all $1 < p < mN/(N-ms)$.
\end{itemize}
\end{theorem}





\begin{proof}[Proof of Theorem \ref{KPZ1existence}]
We use some ideas of \cite[Sections 4 and 6]{AF-2020} and consider separately the cases $m > N/s$ and $1 \leq m \leq N/s$. 

\medbreak \noindent \textbf{Case 1:} $m > N/s$. 
\medbreak
First of all, observe that, without loss of generality, we can assume that 
$N/s < m < 1/(q(t-s)^{+})$. Then,  let us fix $r = r(m,s,t,q) > 0$ such that
$
1 < qm < r < \frac{1}{(t-s)^{+}}
$
and define
\begin{equation} \label{lambdaStarKPZ1}
\lambda_{\star} := \frac{q-1}{q \|f\|_{L^m(\Omega)}} \Big(q  (\widetilde{C}\widetilde{k})^q \|\mu\|_{L^{\infty}(\Omega)} |\Omega|^{\frac{r-qm}{mr}}\Big)^{-\frac{1}{q-1} },
\end{equation}
with $\widetilde{C} > 0$ as in  Proposition \ref{CZ} and $\widetilde{k}$ as in \eqref{emb}. Since $q > 1$, we know (cf. \cite[Lemma 4.1]{AF-2020}) there exists a unique $\ell \in (0,\infty)$ such that
\begin{equation} \label{lKPZ1}
\widetilde{C} \Big(\|\mu\|_{L^{\infty}(\Omega)}  |\O|^{\frac{r-qm }{m r}} \widetilde{k}^q \ell + \lambda_{\star} \|f\|_{L^{m }(\Omega)} \Big) =\ell^{\frac{1}{q}}. 
\end{equation}
Having at hand $\lambda_{\star}$ and $\ell$, we define 
\begin{equation} \label{setE}
E_{\eta} := \left\{ v \in L^{\gamma,1+\eta}_0(\O): \left\| v \right\|_{ L^{\gamma,r}_0(\O)}\le \ell^{\frac{1}{q}}  \right\},
\end{equation}
with 
\begin{equation} \label{gammaEtaEpsilonKPZ1}
\gamma:= \max\{t,s\}, \quad \textup{ and } 0<\eta< \min \left\{q-1, \dfrac{s-N(\gamma-s)}{N(1+(\gamma-s))-s}\right\},
\end{equation}
and point out that $E_{\eta}$ is a closed  and convex subset of $L_0^{\gamma,1+\eta}(\Omega)$. Moreover $E_{\eta}$ is also bounded in $L_0^{\gamma,1+\eta}(\Omega)$. 
Indeed, for any $R>0$, we have that    
\begin{equation} \label{b0}
\big\| (-\Delta)^{\frac{\gamma}{2}} u \big\|_{L^{1+\eta} (\RN)}
\leq
\big\| (-\Delta)^{\frac{\gamma}{2}} u \big\|_{L^{1+\eta} (B_R(0))}
+
\big\| (-\Delta)^{\frac{\gamma}{2}} u \big\|_{L^{1+\eta} (\RN \setminus B_R(0))}\,.
\end{equation}
Then, observe that, since $r>1+\eta$,  
\begin{equation} \label{b1}
\big\| (-\Delta)^{\frac{\gamma}{2}} u \big\|_{L^{1+\eta} (B_R(0))}
\leq 
C_R \big\| (-\Delta)^{\frac{\gamma}{2}} u \big\|_{L^{r} (B_R(0))} 
\leq 
C_R \big\| (-\Delta)^{\frac{\gamma}{2}} u \big\|_{L^{r} (\RN)} 
\leq C_{R} \, \ell^{\frac{1}{q}}\,. 
\end{equation}
On the other hand, choosing $R \geq \frac13 + \frac43 \big(\diam(\Omega) + \dist(0,\Omega)\big)$, we have that $
|x-y| \geq \frac14 \big(1 + |x|\big) $  for all $ y \in \Omega $ and all $x \in \RN \setminus B_R(0)$ and thus, we get that
\begin{equation} \label{b2}
\begin{aligned}
& \big\|(-\Delta)^{\frac{\gamma}{2}} u \big\|_{L^{1+\eta} (\RN \setminus B_R(0))}^{1+\eta} \leq \int_{\RN \setminus B_R(0)} \left| \int_{\Omega} \frac{| u(y)| }{|x-y|^{N+\gamma}  } dy\right|^{1+\eta} dx \\
&  \qquad \leq  C \|u\|_{L^{1+\eta}(\Omega)}^{1+\eta} \int_{\RN \setminus B_R(0)} \frac{dx}{(1+|x|)^{(N+\gamma)(1+\eta)}} \leq \overline{C} \, \|u\|_{L^{1+\eta}(\Omega)}^{1+\eta} \leq \widetilde{C} \, \ell^{\frac{1+\eta}{q}}.
\end{aligned}
\end{equation}
Note that the last inequality follows from the fractional Sobolev inequality (see for instance \cite[Theorem 1.8]{SS-2015}) and the definition of $E_{\eta}$. Gathering \eqref{b0}--\eqref{b2} the boundedness of $E_{\eta}$ in $L_0^{\gamma,1+\eta}(\Omega)$ follows.

To prove the existence  of a weak solution to \eqref{KPZ1} belonging to $E_{\eta}$, we use Schauder's fixed point Theorem. Let us consider
\begin{equation} \label{T1}
T_1: E_{\eta} \to L_0^{\gamma,1+\eta}(\Omega), \quad \varphi \mapsto u,
\end{equation}
where $u$ is the unique weak solution to
\begin{equation} \label{310}
\left\{
\begin{aligned}
(-\Delta)^s u & = \mu(x) |(-\Delta)^{\frac{t}{2}} \varphi|^q + \lambda f(x)\,, \quad  && \textup{ in } \Omega,\\ 
u & = 0, && \textup{ in } \RN \setminus \Omega,
\end{aligned}
\right.
\end{equation}
and observe that, if we prove that $T_1$ has a fixed point in $E_{\eta}$, the existence part immediately follows. Note that, by Proposition \ref{solMap}, the operator $T_1$ is well defined. Hence, to end the proof in this case, we just have to prove that $T_1$ is continuous and compact and that $T_1(E_{\eta}) \subset E_{\eta}$. 

We start proving that $T_1(E_{\eta}) \subset E_{\eta}$. Let $\varphi \in E_{\eta}$ and $u = T_1(\varphi)$. Using Proposition \ref{CZ}, it is immediate to see that $u \in L^{\gamma,1+\eta}_0(\Omega)$ and that
\begin{equation} \label{38}
\begin{aligned}
 \|u\|_{L^{\gamma,r}_0(\O)} & \leq  \widetilde{C} \bigg(\lambda_{\star} \|f\|_{L^{m}(\O)}+ \|\mu\|_{L^{\infty}(\Omega)} \| |(-\Delta)^{\frac{t}{2}}\varphi|^q \|_{L^{m }(\O)}\bigg)  \\
& \leq {\widetilde{C}} \bigg(\lambda_{\star} \|f\|_{L^{m}(\O)}+ \|\mu\|_{L^{\infty}(\Omega)} |\O|^{\frac{r-qm }{rm}}\|(-\Delta)^{\frac{t}{2}}\varphi\|^q_{L^r(\O)}\bigg) \\
& \leq \widetilde{C} \bigg(\lambda_{\star} \|f\|_{L^{m}(\O)}+ \|\mu\|_{L^{\infty}(\Omega)} |\O|^{\frac{r-qm }{rm}}\| \varphi\|^q_{L_0^{t,r}(\Omega)}\bigg)
\\
& \leq  \widetilde{C} \bigg(\lambda_{\star} \|f\|_{L^{m}(\O)}+ \|\mu\|_{L^{\infty}(\Omega)} |\O|^{\frac{r-qm }{rm}} \widetilde{k}^q \| \varphi\|^q_{L^{\gamma,r}_0(\Omega)}\bigg)
\leq \ell^{\frac{1}{q}}\,.
\end{aligned}
\end{equation}
Hence, it follows that $T_1(E_{\eta}) \subset E_{\eta}$. 

 Next, we prove that $T_1$ is compact. Let $(\varphi_n)_n \subset E_{\eta}$ be such that $\|\varphi_n\|_{L^{\gamma,1+\eta}_0(\Omega)} \leq 1$ for all $n \in \N$. Also, let $h_n := \mu(x) |(-\Delta)^{\frac{t}{2}} \varphi_n|^q + \lambda f(x)$ for all $n \in \N$. Arguing as in \eqref{38}, it is immediate to check that $(h_n)_n$ is bounded in $L^1(\Omega)$ and thus, the compactness of $T_1$ immediately follows from Proposition \ref{solMap}. 
 
 Finally, we prove that $T_1$ is continuous. Let $(\varphi_n)_n \subset E_{\eta}$ be a sequence such that $\varphi_n \to \varphi$ in $L_0^{\gamma,1+\eta}(\Omega)$ and let $u_n = T_1(\varphi_n)$ for all $n \in \N$ and $u = T_1(\varphi)$. Note that
\begin{equation} \label{previous}
\left\{
\begin{aligned}
(-\Delta)^s (u_n-u) & = \mu(x) \Big( |(-\Delta)^{\frac{t}{2}} \varphi_n|^q - |(-\Delta)^{\frac{t}{2}} \varphi|^q \Big), \quad && \textup{ in } \Omega,\\
u_n - u & = 0, && \textup{ in } \RN \setminus \Omega.
\end{aligned}
\right.
\end{equation}
If we show that the $L^1$--norm of the right hand side in the \eqref{previous} goes to $0$ as $n \to \infty$, the continuity of $T_1$ immediately follows from Proposition \ref{solMap} and the existence in the case where $m > N/s$ will follow.  By direct computations (using H\"older inequality and the Mean value Theorem), it follows that
\begin{equation} \label{step1ContinuityT1}
\Big\|  \mu(x) \Big( |(-\Delta)^{\frac{t}{2}} \varphi_n|^q - |(-\Delta)^{\frac{t}{2}} \varphi|^q \Big) \Big\|_{L^1(\Omega)} \leq C\, \|(-\Delta)^{\frac{t}{2}}(\varphi_n - \varphi)\|_{L^q(\Omega)}, 
\end{equation}
for some $C > 0$ depending only on $q$, $r$, $\|\mu\|_{L^{\infty}(\Omega)}$, $|\Omega|$ and $\ell$. On the other hand, using \eqref{emb} and Littlewood's inequality (or interpolation in $L^p$--spaces), we infer that
\begin{equation} \label{step2ContinuityT1}
\begin{aligned}
& \|(-\Delta)^{\frac{t}{2}}(\varphi_n-\varphi)\|_{L^q(\Omega)} \leq \|\varphi_n - \varphi\|_{L_0^{t,q}(\Omega)} \leq \widetilde{k} \|\varphi_n - \varphi\|_{L_0^{\gamma,q}(\Omega)} =  \widetilde{k} \|(-\Delta)^{\frac{\gamma}{2}}(\varphi_n - \varphi)\|_{L^q(\RN)}\\
& \quad \leq  \widetilde{k}  \|(-\Delta)^{\frac{\gamma}{2}}(\varphi_n - \varphi)\|_{L^{1+\eta}(\RN)}^{\theta}  \|(-\Delta)^{\frac{\gamma}{2}}(\varphi_n - \varphi)\|_{L^r(\RN)}^{1-\theta} =  \widetilde{k} \|\varphi_n - \varphi\|_{L_0^{\gamma,1+\eta}(\Omega)}^{\theta} \|\varphi_n - \varphi\|_{L_0^{\gamma,r}(\Omega)}^{1-\theta} \\
& \quad \leq (2\ell^{\frac{1}{q}})^{1-\theta}  \widetilde{k} \|\varphi_n - \varphi\|_{L_0^{\gamma,1+\eta}(\Omega)}^{\theta},
\end{aligned}
\end{equation}
with $\frac{1}{q} = \frac{\theta}{1+\eta}+\frac{1-\theta}{r}$. Since $\varphi_n \to \varphi$ in $L_0^{\gamma,1+\eta}(\Omega)$, combining \eqref{step1ContinuityT1} and \eqref{step2ContinuityT1}, we conclude that
$$
\lim_{n \to \infty} \Big\|  \mu \Big( |(-\Delta)^{\frac{t}{2}} \varphi_n|^q - |(-\Delta)^{\frac{t}{2}} \varphi|^q \Big) \Big\|_{L^1(\Omega)} = 0,
$$
as desired. The proof of the existence in the case where $m > N/s$ is thus finished. Once we have the existence of a weak solution $u \in E_{\eta}$, the claimed regularity immediately follows from the definition of $E_{\eta}$, our choice of $r$, Proposition \ref{CZ} and \cite[Proposition 1.4 (iii)]{RO-S14}.

\medbreak \noindent \textbf{Case 2:} $1 \leq m \leq N/s$.
\medbreak
First of all, let us fix $r = r(m,s,t,q) > 0$ such that $1 < qm < r < \frac{mN}{(N-ms+mN(\gamma-s))^{+}}$ and consider $\lambda_{\star}$ and $\ell$ as in \eqref{lambdaStarKPZ1} and \eqref{lKPZ1} respectively. Then, we define
\begin{equation}
\widetilde{E}_{\eta} := \Big\{ v \in L_0^{\gamma,1+\eta}(\Omega) : \|v\|_{L_0^{\gamma,1+\eta}(\Omega)} \leq M \textup{ and } \|v\|_{L_0^{\gamma,r}(\Omega)} \leq \ell^{\frac{1}{q}} \Big\},
\end{equation}
with $\gamma$ and $\eta$ as in \eqref{gammaEtaEpsilonKPZ1} and
$$
M:= \tilde{C}_1 \Big( \|\mu\|_{L^{\infty}(\Omega)}|\Omega|^{\frac{r-q}{r}} \widetilde{k}^q \ell + \lambda_{\star} \|f\|_{L^1(\Omega)} \Big),
$$
where $\tilde{C}_1 > 0$ is the constant that appears in Proposition \ref{CZ} for $m=1$. It is immediate to see that $\widetilde{E}_{\eta}$ is a bounded, closed and convex set of $L_0^{\gamma,1+\eta}(\Omega)$. Moreover, arguing as in the first case, one can prove that
\begin{equation} \label{T1tilde}
\widetilde{T}_1: \widetilde{E}_{\eta} \to L_0^{\gamma,1+\eta}(\Omega), \quad \varphi \mapsto u,
\end{equation}
where $u$ is the unique weak solution to \eqref{310}, is well-defined, continuous, compact and satisfies $\widetilde{T}_1(\widetilde{E}_{\eta}) \subset \widetilde{E}_{\eta}$. Hence, applying again Schauder's fixed point Theorem, the existence follows also in this case. Having at hand the existence, the claimed regularity follows again from Proposition \ref{CZ}.
\end{proof}

Next, we analyse the existence of weak solution to \eqref{KPZ} in the case where $\mathbb{D}_t = \nabla^t$. More precisely, we analyse the existence of weak solution to 
\begin{equation}  \tag{KPZ$_2$}\label{KPZ2} 		
\left\{ 		
\begin{aligned} 			
(-\Delta)^s u & = \mu(x) |\nabla^t u|^q + \lambda f(x), && \quad \textup{ in } \Omega,\\ 
u & = 0, && \quad \textup{ in }   \RN \setminus \Omega.\\
\end{aligned} 		
\right. 	
\end{equation} 	


Our main result concerning \eqref{KPZ2} can be formulated as follows: 

\begin{theorem} \label{KPZ2existence}
Assume that $0 < t < \min\{1,s(1+N^{-1})\}$, $f \in L^m(\Omega)$ for some $m \geq 1$ and $\mu \in L^{\infty}(\Omega)$. Then, for all $1 < q < \overline{q}$, there exists $\lambda_{\star} > 0$ such that, for all $0 < \lambda \leq \lambda_{\star}$, \eqref{KPZ2} has a weak solution $u$. Moreover:
\begin{itemize}
\item[$\circ$] If $m \geq N/s$, then $u \in W^{s,p}(\RN) \cap C^{0,s}(\RN)$ for all $1 < p < + \infty$. 
\item[$\circ$] If $1 \leq m < N/s$, then $u\in W^{s,p}(\RN)$ for all $1 < p < mN/(N-ms)$.
\end{itemize}
\end{theorem}

\begin{remark}Let us emphasize that, having at hand \cite[Theorem 1.7]{SS-2015}, the proof of Theorem \ref{KPZ2existence} follows arguing exactly as in the proof of Theorem \ref{KPZ1existence}.
\end{remark} 

\bigbreak
Finally, we analyse the existence of weak solution in the slightly more involved case where $\mathbb{D}_t = \mathcal{D}_t$. More precisely, we analyse the existence of weak solution to

\begin{equation}  \tag{KPZ$_3$}\label{KPZ3} 		
\left\{ 		
\begin{aligned} 			
(-\Delta)^s u & = \mu(x) (\mathcal{D}_t u)^q + \lambda f(x), && \quad \textup{ in } \Omega,\\ 
u & = 0, && \quad \textup{ in }   \RN \setminus \Omega.\\
\end{aligned} 		
\right. 	
\end{equation}

\noindent Setting $\gamma:= \max\{t,s\}$ and 
$$
\overline{m}(s,t):= \frac{2N}{N+2s-2N(t-s)^+}
$$
our main result concerning \eqref{KPZ3} reads as follows:

\begin{theorem} \label{KPZ3existence}
Assume that $0 < t < \min\{1,s(1+N^{-1})\}$, $f \in L^m(\Omega)$ for some $m > \overline{m}$ and $\mu \in L^{\infty}(\Omega)$. Then, for all $1 < q < \overline{q}$, there exists $\lambda_{\star} > 0$ such that, for all $0 < \lambda \leq \lambda_{\star}$, \eqref{KPZ3} has a weak solution $u$. Moreover:
\begin{itemize}
\item[$\circ$] If $m \geq N/s$, then $u \in W^{s,p}(\RN) \cap C^{0,s}(\RN)$ for all $1 < p < + \infty$. 
\item[$\circ$] If $\overline{m} < m < N/s$, then $u\in W^{s,p}(\RN)$ for all $1 < p < mN/(N-ms)$.
\end{itemize}
\end{theorem}

\begin{proof}
We consider separately the cases $m > N/s$ and $\overline{m} < m \leq N/s$.  
\medbreak
\noindent \textbf{Case 1:} $m > N/s$.
\medbreak
First of all, observe that, without loss of generality, we can assume that $N/s < m < 1/(q(t-s)^{+})$. Then,  let us fix $r = r(m,s,t,q) > 0$ such that
$
\max\{qm,2\} < r < \frac{1}{(t-s)^{+}}
$
and define
\begin{equation} \label{lambdaStarKPZ3}
\lambda_{\star} := \frac{q-1}{q \|f\|_{L^m(\Omega)}} \Big(q  (\widetilde{C}\widetilde{k})^q \|\mu\|_{L^{\infty}(\Omega)} |\Omega|^{\frac{r-qm}{mr}}\Big)^{-\frac{1}{q-1} },
\end{equation}
with $\widetilde{C} > 0$ as in Proposition \ref{CZ} and $\widetilde{k}$ as in \eqref{emb}. Since $q > 1$, we know (cf. \cite[Lemma 4.1]{AF-2020}) there exists a unique $\ell \in (0,\infty)$ such that
\begin{equation} \label{lKPZ3}
\widetilde{C} \Big(\|\mu\|_{L^{\infty}(\Omega)}  |\O|^{\frac{r-qm }{m r}} \widetilde{k}^q \ell + \lambda_{\star} \|f\|_{L^{m }(\Omega)} \Big) =\ell^{\frac{1}{q}}. 
\end{equation}
Having at hand $\lambda_{\star}$ and $\ell$, we define 
\begin{equation} \label{setE}
E_{\eta} := \left\{ v \in L^{\gamma,1+\eta}_0(\O): \vertiii{v}_{ L^{\gamma,r}_0(\O)}\le \ell^{\frac{1}{q}}  \right\},
\end{equation}
with 
\begin{equation} \label{gammaEtaEpsilonKPZ3}
\gamma= \max\{t,s\} \quad \textup{ and } \quad 0<\eta< \min \left\{q-1, \dfrac{s-N(\gamma-s)}{N(1+(\gamma-s))-s}\right\},
\end{equation}
and we point out that $E_{\eta}$ is a closed, bounded and convex subset of $L_0^{\gamma,1+\eta}(\Omega)$. 
To prove the existence of a weak solution to \eqref{KPZ3} belonging to $E_{\eta}$,  we use again Schauder's fixed point Theorem. Let us define
\begin{equation} \label{T2}
T_2: E_{\eta} \to L_0^{\gamma,1+\eta}(\Omega), \quad \varphi \mapsto u,
\end{equation}
where $u$ the unique weak solution to
\begin{equation} \label{317}
\left\{
\begin{aligned}
(-\Delta)^s u & = \mu(x) (\mathcal{D}_t\varphi)^q + \lambda f(x)\,, \quad  && \textup{ in } \Omega,\\ 
u & = 0, && \textup{ in } \RN \setminus \Omega.
\end{aligned}
\right.
\end{equation}
Note that, by Proposition \ref{solMap}, the operator $T_2$ is well defined. Hence, to conclude the proof in this case, we just have to prove that $T_2$ is continuous and compact and that $T_2(E_{\eta}) \subset E_{\eta}$. The compactness of $T_2$ and the fact that $T_2(E_{\eta}) \subset E_{\eta}$ can be proved arguing exactly as in the proof of Theorem \ref{KPZ1existence}. However, to prove that $T_2$ is continuous we have to argue in a different way. Let $(\varphi_n)_n \subset E_{\eta}$ be a sequence such that $\varphi_n \to \varphi$ in $L_0^{\gamma,1+\eta}(\Omega)$ and let $u_n = T_2(\varphi_n)$ for all $n \in \N$ and $u = T_2(\varphi)$. Note that
\begin{equation} \label{previousKPZ3}
\left\{
\begin{aligned}
(-\Delta)^s (u_n - u) & = \mu(x) \Big( (\mathcal{D}_t\varphi_n)^q - (\mathcal{D}_t\varphi)^q \Big), \quad && \textup{ in } \Omega,\\
u_n - u & = 0, && \textup{ in } \RN \setminus \Omega.
\end{aligned}
\right.
\end{equation}
If we show that the $L^1$--norm of the right hand side in \eqref{previousKPZ3} goes to $0$ as $n \to \infty$, the continuity of $T_2$ immediately follows from Proposition \ref{solMap}. 
We actually prove something more general from which the existence part of the result in the case $m > N/s$ immediately follows. 
\medbreak
\noindent \textbf{Claim:} \textit{For all $1 < \alpha < r$, it follows that}
\begin{equation} \label{claimConvergence}
\lim_{n \to \infty} \int_{\Omega} \big|(\mathcal{D}_t\varphi_n(x))^{\alpha}-(\mathcal{D}_t\varphi(x))^{\alpha} \big| dx = 0.
\end{equation}

\noindent \emph{Proof of the claim.} First of all, let $2 \leq \beta < r$ be fixed but arbirary. Using the Mean value Theorem and H\"older and triangular inequalities, one can easily get that
\begin{align*}
& \int_{\Omega} |(\mathcal{D}_t\varphi_n(x))^{\beta}-(\mathcal{D}_t\varphi(x))^{\beta} \big| dx 
= \int_{\Omega} \big|\big((\mathcal{D}_t\varphi_n(x))^2\big)^{\frac{\beta}{2}}-\big((\mathcal{D}_t\varphi(x))^2\big)^{\frac{\beta}{2}} \big| dx \\
& \quad \leq \frac{\beta}{2} \int_{\Omega} \big| (\mathcal{D}_t\varphi_n(x))^2-  (\mathcal{D}_t\varphi(x))^2 \big| \big(  (\mathcal{D}_t\varphi_n(x))^2 +  (\mathcal{D}_t\varphi(x))^2 \big)^{\frac{\beta-2}{2}} dx \\
&  \quad \leq \frac{\beta}{2} \int_{\Omega} (\mathcal{D}_t(\varphi_n-\varphi)(x)) (\mathcal{D}_t(\varphi_n+\varphi)(x))  \big(  (\mathcal{D}_t\varphi_n(x))^2 +  (\mathcal{D}_t\varphi(x))^2 \big)^{\frac{\beta-2}{2}} dx \\
& \quad \leq \frac{\sqrt{2}}{2} \beta \int_{\Omega}  (\mathcal{D}_t(\varphi_n-\varphi)(x)) \big( \mathcal{D}_t\varphi_n(x) + \mathcal{D}_t \varphi(x)  \big)^{\beta-1} dx \leq   \frac{\sqrt{2}}{2} \beta \|\mathcal{D}_t(\varphi_n - \varphi) \|_{L^\beta(\Omega)} \|\mathcal{D}_t\varphi_n + \mathcal{D}_t \varphi\|_{L^\beta(\Omega)}^{\beta-1} \\
&  \quad \leq C \big(  \|\mathcal{D}_t\varphi_n\|_{L^\beta(\Omega)}^{\beta-1} +  \| \mathcal{D}_t \varphi \|_{L^\beta(\Omega)}^{\beta-1} \big)  \|\mathcal{D}_t(\varphi_n - \varphi) \|_{L^\beta(\Omega)} \leq C \big( \vertiii{\varphi_n}_{L_0^{t,\beta}(\Omega)}^{\beta-1} + \vertiii{\varphi}_{L_0^{t,\beta}(\Omega)}^{\beta-1} \big) \vertiii{\varphi_n-\varphi}_{L_0^{t,\beta}(\Omega)}, 
\end{align*}
for some $C > 0$ depending only on $\beta$. Then, arguing exactly as in the proof of \eqref{step2ContinuityT1}, we conclude that
$$
\lim_{n \to \infty} \int_{\Omega} \big|(\mathcal{D}_t\varphi_n(x))^{\beta}-(\mathcal{D}_t\varphi(x))^{\beta} \big| dx = 0.
$$
Since $\beta \in [2,r)$ was fixed but arbitrary, we have proved \eqref{claimConvergence} for all $2 \leq \alpha < r$. 

It remains to deal with the case $1 < \alpha < 2$. To that end, note that $(\mathcal{D}_t\varphi_n)_n \subset L^2(\Omega)$ is a bounded non-negative sequence such that $\|\mathcal{D}_t\varphi_n\|_{L^2(\Omega)} \to \|\mathcal{D}_t \varphi\|_{L^2(\Omega)}$ as $n \to \infty$. Hence, it follows that $\mathcal{D}_t\varphi_n \to \mathcal{D}_t \varphi$ in $L^2(\Omega)$ as $n \to \infty$. Also, observe that, for any   $\alpha \in (1,2)$, 
\begin{align*}
& \int_{\Omega} \big|(\mathcal{D}_t\varphi_n(x))^{\alpha}-(\mathcal{D}_t\varphi(x))^{\alpha} \big| dx \\
& \quad \leq C\, \|\mathcal{D}_t(\varphi_n)-\mathcal{D}_t(\varphi)\|_{L^{\alpha}(\Omega)} \big( \|\mathcal{D}_t\varphi_n\|_{L^{\alpha}(\Omega)}^{\alpha-1} + \|\mathcal{D}_t\varphi\|_{L^{\alpha}(\Omega)}^{\alpha-1} \big) \leq \widetilde{C}\, \|\mathcal{D}_t \varphi_n - \mathcal{D}_t \varphi\|_{L^2(\Omega)},
\end{align*}
with $C > 0$ depending only on $\alpha > 0$ and $\widetilde{C} > 0$ depending only on $\alpha$, $r$, $|\Omega|$ and $\ell$. Combining this chain of inequalities with the fact that $\mathcal{D}_t\varphi_n \to \mathcal{D}_t \varphi$ in $L^2(\Omega)$ as $n \to \infty$ we conclude that
$$
\lim_{n \to \infty} \int_{\Omega} \big|(\mathcal{D}_t\varphi_n(x))^{\alpha}-(\mathcal{D}_t\varphi(x))^{\alpha} \big| dx = 0.
$$
Once the claim is proved, to conclude the proof in the case where $m > N/s$, it just remains to prove the claimed regularity, that follows thanks to Proposition \ref{CZ}, \cite[Proposition 1.4 (iii)]{RO-S14} and the definition of  $E_{\eta}$.
\medbreak
\noindent \textbf{Case 2:} $\overline{m} < m \leq  N/s$. 
\medbreak
First of all, let us fix $r = r(m,s,t,q) > 0$ such that $\max\{2,qm\} < r < \frac{mN}{(N-ms+mN(\gamma-s))^{+}}$ and consider $\lambda_{\star}$ and $\ell$ as in \eqref{lambdaStarKPZ3} and \eqref{lKPZ3} respectively. Then, let us define
\begin{equation}
\widetilde{E}_{\eta} := \Big\{ v \in L_0^{\gamma,1+\eta}(\Omega) : \vertiii{v}_{L_0^{\gamma,1+\eta}(\Omega)} \leq M \textup{ and } \vertiii{v}_{L_0^{\gamma,r}(\Omega)} \leq \ell^{\frac{1}{q}} \Big\},
\end{equation}
with $\gamma$ and $\eta$ as in \eqref{gammaEtaEpsilonKPZ3} and
$$
M:= C \Big( \|\mu\|_{L^{\infty}(\Omega)}|\Omega|^{\frac{r-q}{r}} \widetilde{k}^q \ell + \lambda_{\star} \|f\|_{L^1(\Omega)} \Big),
$$
with $C > 0$ depending only on $\Omega$, $N$, $s$ and $t$ (cf. Proposition \ref{CZ}). It is easy to check that $\widetilde{E}_{\eta}$ is a bounded, closed and convex set of $L_0^{\gamma,1+\eta}(\Omega)$. Moreover, arguing as we did in the first case, one can prove that
\begin{equation} \label{T1tilde}
\widetilde{T}_2: \widetilde{E}_{\eta} \to L_0^{\gamma,1+\eta}(\Omega), \quad \varphi \mapsto u,
\end{equation}
where $u$ is the unique weak solution to \eqref{317}, is well-defined, continuous, compact and satisfies $\widetilde{T}_2(\widetilde{E}_{\eta}) \subset \widetilde{E}_{\eta}$. Hence, applying again Schauder's fixed point Theorem, the existence follows also in this case. Having at hand the existence, the claimed regularity follows again from Proposition \ref{CZ}.
\end{proof}


\begin{remark}
The appearance of the upper bound $\overline{q} \in (1,+\infty]$ in Theorems \ref{KPZ1existence} and \ref{KPZ3existence} is natural in this kind of existence results. Even in the local case \eqref{local}, one finds an upper bound on $q$ depending on the regularity of the data $f$ and on the dimension $N$. In our proofs, the upper bound $\overline{q}$ naturally appears when applying Proposition \ref{CZ} and depends on the regularity of the datum $f$, on the dimension $N$ and on the order of the ``nonlocal gradient''. The lower bound $\overline{m}$ is again related to the use of Proposition \ref{CZ} but also to the fact that we need $r > 2$ in the proof of Theorem \ref{KPZ3existence}.
\end{remark}

\section{Non--existence results} \label{nonExistenceResults}

\noindent This section is devoted to prove non-existence results for \eqref{KPZ1} and \eqref{KPZ3}.  Let us start analysing \eqref{KPZ1}. As already mentioned in the introduction, the proof of the non-existence result for \eqref{KPZ1} is completely different from its counterpart for \eqref{KPZ3}. Here, we prove a generalization of Theorem \ref{nonexist-intro}.

\begin{theorem} \label{nonexist}
Let $0 < t < \min\{1,2s\}$, $\mu_2 \geq \mu(x) \geq \mu_1 > 0$, $f\in L^1 (\O)$  with $f \gneqq 0$ and $q> \frac{2(s+1)}{t+2}$. Then, there exists $\lambda^{\star \star} > 0$ such that, for all
$ \lambda > \lambda^{\star \star}$, \eqref{KPZ1} has no weak solution in $W_0^{s,2}(\Omega)$.
\end{theorem}

\begin{proof}
Let $u \in W_0^{s,2}(\Omega)$ be a weak  solution to \eqref{KPZ1} and let $\phi \in W_0^{s,2}(\Omega) \cap C^s(\RN)$ be the unique (energy) solution to
$$
\left\{
\begin{aligned}
(-\Delta)^s \phi & = 1, \quad && \textup{ in } \Omega,\\
\phi & = 0, && \textup{ in } \RN \setminus \Omega.
\end{aligned}
\right.
$$
First of all, note that $\phi \in \mathbb{X}^s(\Omega)$, so that we can use $\phi$ as test function in \eqref{KPZ1} and get that
\begin{equation} \label{41}
\begin{aligned}
\int_{\Omega} u \,dx =  \int_{\Omega} \mu(x) |(-\Delta)^{\frac{t}{2}} u|^q \phi\, dx + \lambda \int_{\Omega} f \phi\, dx \geq \mu_1 \int_{\Omega} |(-\Delta)^{\frac{t}{2}} u|^q \phi \, dx + \lambda \int_{\Omega} f \phi\, dx.
\end{aligned}
\end{equation}
On the other hand, let $\psi \in W_0^{\frac{t}{2},2}(\Omega) \cap C^{ \frac{t}{2}}(\RN)$ be the unique (energy) solution to the problem
 \begin{equation} \label{laps2}
\left\{
\begin{aligned}
(-\Delta)^{\frac{t}{2}} \psi & = 1, \quad && \textup{ in } \Omega,\\
\psi & = 0, && \textup{ in } \RN \setminus \Omega.
\end{aligned}
\right.
\end{equation}
Since  $ W_0^{s,2}(\Omega) \subset W_0^{\frac{t}{2},2}(\Omega)$ we can  test \eqref{laps2}  with $u$ and integrate by parts, so that, using \eqref{41},  we obtain  
\begin{equation} \label{43}
\int_{\Omega} \psi (-\Delta)^{\frac{t}{2}}u\, dx = \int_{\Omega} u\, dx \geq \mu_1 \int_{\Omega} |(-\Delta)^{\frac{t}{2}} u|^q \phi \, dx + \lambda \int_{\Omega} f \phi\, dx.
\end{equation}
Moreover, using Young's inequality, we easily see that
\begin{equation} \label{44}
\int_{\Omega} \psi  (-\Delta)^{\frac{t}{2}}u\, dx \leq \mu_1 \int_{\Omega} |(-\Delta)^{\frac{t}{2}}u|^q \phi\, dx +   {C_q}\mu_1^{-\frac1{q-1}} \int_{\Omega} \frac{\psi^{\frac{q}{q-1}}}{\phi^{\frac{1}{q-1}}} dx.
\end{equation}
Thus, combining \eqref{43} and \eqref{44}, we obtain that   
\begin{equation} \label{45}
\lambda \int_{\Omega} f \phi\, dx \leq   {C_q}\mu_1^{-\frac1{q-1}} \int_{\Omega} \frac{\psi^{\frac{q}{q-1}}}{\phi^{\frac{1}{q-1}}} dx.
\end{equation}
Finally, note that (see e.g. \cite[Eq. (1.15)]{F-J-2021}) there exists $C_0 >   0$ (depending only on $\Omega$, $s,t$ and $N$) such that
\begin{equation} \label{455}
C_0^{-1} \delta^s(x) \leq \phi(x) \leq C_0 \delta^{s}(x) \quad \textup{ and } \quad C_0^{-1} \delta^{\frac{t}{2}}(x) \leq \psi(x) \leq C_0 \delta^{\frac{t}{2}}(x), \quad \textup{ in }\ \overline{\Omega}.
\end{equation}
Hence, since $q > \frac{2(s+1)}{t + 2}$, the right hand side in \eqref{45} is bounded, an thus necessarily 
\begin{equation} \label{46}
\lambda   
\leq  
\frac{ C_q \displaystyle \int_{\Omega} \frac{\psi^{\frac{q}{q-1}}}{\phi^{\frac{1}{q-1}}} dx.
}{\displaystyle {\mu_1^{\frac1{q-1}}} \int_{\Omega} f \phi\, dx} =: \lambda^{\star \star}.
\end{equation}
\end{proof}

\begin{remark}
The regularity imposed on $f$ can be slightly weakened. Indeed, to prove our non--existence result, namely Theorem \ref{nonexist}, we only need $f \in L^1(\Omega, \delta^s(x)dx)$. 
Moreover, observe that in the proof we only used that $u$ is a (weak) supersolution to \eqref{KPZ1}. On the negative side,  observe that the bound appearing on the power $q$ seems to be technical.
\end{remark}

We also prove here a generalization of \cite[Theorem 1.2]{AF-2020} covering the cases where $q \neq 2$ and/or $t \neq s$.

\begin{theorem} 
Let $0 < t < \min\{1,2s\}$, $\mu_2 \geq \mu(x) \geq \mu_1 > 0$, $f \in L^1(\Omega)$ with $f^{+} \not \equiv 0$ and $q > 1$. Then, there exists $\lambda^{\star \star} > 0$ such that, for all $\lambda > \lambda^{\star \star}$, \eqref{KPZ3} has no weak solution in $W_0^{s,2}(\Omega)$.
\end{theorem}

\begin{proof}
Let $u \in W_0^{s,2}(\Omega)$ be a weak solution to \eqref{KPZ3} and let $\phi \in C_c^{\infty}(\Omega)$ be an arbitrary non-negative function such that
\begin{equation} \label{propphi}
\int_{\Omega} f \phi^{\frac{q}{q-1}} dx > 0 \quad \textup{ and } \quad \int_{\Omega}  \big( \mathcal{D}_{2s-t}\phi \big)^{\frac{q}{q-1}} dx < + \infty. 
\end{equation}
Using $\phi^{\frac{q}{q-1}}$ as test function in \eqref{KPZ3}, we get
\begin{equation}  \label{48}
\int_{\Omega} u (-\Delta)^s (\phi^{\frac{q}{q-1}})\, dx \geq \mu_1 \int_{\Omega} (\mathcal{D}_t u )^q \phi^{\frac{q}{q-1}} dx + \lambda \int_{\Omega} f \phi^{\frac{q}{q-1}} dx.
\end{equation}
On the other hand, using the Mean value Theorem and H\"older's and Young's inequalities, we infer that
\begin{align*}
& \int_{\Omega} u  (-\Delta)^s (\phi^{\frac{q}{q-1}})\, dx = \iint_{\R^{2N}} \frac{(u(x)-u(y))(\phi^{\frac{q}{q-1}}(x) - \phi^{\frac{q}{q-1}}(y))}{|x-y|^{N+2s}} \,dy dx  \\
& \quad \leq C_q \iint_{\R^{2N}} \frac{|u(x)-u(y)||\phi(x)-\phi(y)|}{|x-y|^{N+2s}} (\phi^{\frac{1}{q-1}}(x) + \phi^{\frac{1}{q-1}}(y))\, dy dx \\
& \quad = 2 C_q \iint_{\R^{2N}} \frac{|u(x)-u(y)||\phi(x)-\phi(y)|}{|x-y|^{N+2s}} \phi^{\frac{1}{q-1}}(x)\, dy dx 
\leq 2 C_q \int_{\RN} \big(\mathcal{D}_t  u\big)\, \big( \mathcal{D}_{2s-t} \phi\big)\  \phi^{\frac{1}{q-1}}\, dx \\
& \quad = 2 C_q \int_{\Omega} \big(\mathcal{D}_t u\big)  \, \big(\mathcal{D}_{2s-t} \phi\big) \phi^{\frac{1}{q-1}} \, dx \leq \mu_1 \int_{\Omega} (\mathcal{D}_t u )^q \phi^{\frac{q}{q-1}} dx + \widetilde{C}_{q,\mu_1} \int_{\Omega} (\mathcal{D}_{2s-t}\phi)^{\frac{q}{q-1}} dx\,.
\end{align*}
Combining the above chain of inequalities with \eqref{48}, we get that necessarily 
$$
\lambda \int_{\Omega} f \phi^{\frac{q}{q-1}} dx \leq  \widetilde{C}_{q,\mu_1} \int_{\Omega} (\mathcal{D}_{2s-t}\phi)^{\frac{q}{q-1}} dx.
$$
Hence, defining
$$
\lambda^{\star \star} := \inf \left\{ \frac{ \displaystyle \widetilde{C}_{q,\mu_1} \int_{\Omega} (\mathcal{D}_{2s-t}\phi)^{\frac{q}{q-1}} dx}{ \displaystyle \int_{\Omega}f \phi^{\frac{q}{q-1}} dx} : \phi \in C_c^{\infty}(\Omega) \textup{ is non-negative and satisfies  }  \eqref{propphi} \right\},
$$
the result immediately follows. 
\end{proof}

\bibliographystyle{plain}
\bibliography{Bibliography}
\vspace{0.5cm}

\end{document}